\documentclass{amsart}

\usepackage{amssymb}
\usepackage[all]{xy}
\usepackage{hyperref}

\usepackage{enumitem}   

\setlist[enumerate]{itemsep=.2em,topsep=.2em,leftmargin=1.25em,itemindent=2.0em}

\newtheorem{thm}{Theorem}

\newtheorem{lem}[thm]{Lemma}
\newtheorem{cor}[thm]{Corollary}

\newtheorem{conj}[thm]{Conjecture}
   
\theoremstyle{definition}
\newtheorem{defn}[thm]{Definition}

\newtheorem{say}[thm]{}
\newtheorem{exmp}[thm]{Example}

\newtheorem*{ack}{Acknowledgments}      

\newtheorem{defn-thm}[thm]{Definition--Theorem}  
\newtheorem{defn-lem}[thm]{Definition--Lemma}  

\theoremstyle{remark}

\renewcommand{\c}[0]{{\mathbb C}}  

\renewcommand{\o}[0]{{\mathcal O}} 
\newcommand{\z}[0]{{\mathbb Z}}

\renewcommand{\a}[0]{{\mathbb A}} 

\newcommand{\dd}[0]{{\mathbb D}}

\newcommand{\p}[0]{{\mathbb P}}

\newcommand{\q}[0]{{\mathbb Q}}
\newcommand{\map}[0]{\dasharrow}
\newcommand{\qtq}[1]{\quad\mbox{#1}\quad}

\newcommand{\pic}[0]{\operatorname{Pic}}

\newcommand{\proj}[0]{\operatorname{Proj}}

\newcommand{\rdown}[1]{\lfloor{#1}\rfloor}

\newcommand{\onto}[0]{\twoheadrightarrow}

\newcommand{\tsum}[0]{\textstyle{\sum}}

\newcommand{\defor}[0]{\operatorname{Def}}

\def\into{\DOTSB\lhook\joinrel\to}

\def\loccoh#1.#2.#3.#4.{H^{#1}_{#2}(#3,#4)}

\DeclareMathAlphabet{\mathchanc}{OT1}{pzc}%
                                {m}{it}

\usepackage[all]{xy}\xyoption{dvips}

\begin{document}
\bibliographystyle{amsalpha}

  \title{Deformations of  varieties of general type}
  \author{J\'anos Koll\'ar}

\begin{abstract} 
We prove that  small deformations of a projective variety of general type   are also  projective varieties of general type,  with the same plurigenera.
\end{abstract}

 \maketitle

Our aim is to prove the following.

\begin{thm} \label{mmp.extends.conj.thm.c.cor}
Let $g:X\to S$ be a flat, proper morphism of complex analytic spaces. 
Fix a point $0\in S$ and assume that the fiber
$X_0$ is  projective, of general type, and with  canonical  singularities.
Then   there is an open  neighborhood  $0\in U\subset S$ such that
\begin{enumerate}
\item the plurigenera of $X_s$  are independent of $s\in U$ for every $r$, and
\item  the fibers $X_s$ are  projective  for every $s\in U$.
\end{enumerate}
\end{thm}

Here the $r$th plurigenus of $X_s$  is
$h^0(Y_s, \omega_{Y_s}^{r})$, where $Y_s\to X_s$ is any resolution of $X_s$. By \cite[VI.5.2]{nak-book}  (see also (\ref{defs.of.can.lem}.2)) $X_s$ has canonical  singularities, so this is the same as 
 $h^0(X_s, \omega_{X_s}^{[r]})$, where $\omega_{X_s}^{[r]}$ denotes the double dual of the $r$th tensor power $\omega_{X_s}^{\otimes r}$.

\medskip

{\it Comments \ref{mmp.extends.conj.thm.c.cor}.3.}
Many  cases of this have been proved, but
I believe that the general result is new,  even  for $X_0$ smooth and $S$ a disc.

For smooth surfaces proofs are given in  \cite{MR0112154, Iitaka69}, and for 3-folds 
with  terminal singularities in 
\cite[12.5.1]{km-flips}.
If $g$ is assumed projective, then of course all fibers are projective,
and deformation invariance of plurigenera was proved by \cite{siu-plur}  for $X_0$ smooth, and by \cite[Chap.VI]{nak-book} when $X_0$ has  canonical  singularities.
However,  frequently $g$ is not projective;  see Example~\ref{surf.fam.exmp}
for some smooth, 2-dimensional examples.
Many projective varieties have deformations that are not projective, not even algebraic in any sense; K3 and elliptic surfaces furnish the best known examples.

In Example~\ref{amp.K.quot.sing.exmp} we construct a  deformation of 
a projective surface with a quotient singularity and ample canonical class, whose general fibers are   non-algebraic, smooth surfaces  of Kodaira dimension 0. Thus canonical is likely the largest class of singularities where
Theorem~\ref{mmp.extends.conj.thm.c.cor} holds. See also Example~\ref{bad.fams.exmp} for surfaces with simple elliptic singularities. 

The projectivity of $X_0$ is essential in our proof, but (\ref{mmp.extends.conj.thm.c.cor}.1) should hold whenever $X_0$ is a proper algebraic space of general type with  canonical  singularities. Such results are proved in 
\cite{rao-tsa}, provided one assumes that  either $X_0$ is smooth and all fibers are Moishezon, or almost all fibers are  of general type.

Our main technical result says that the Minimal Model Program  works for $g:X\to S$.
For $\dim X_0=2$ and $X_0$ smooth, this goes back to \cite{MR0112154}. 
For $\dim X_0=3$ and terminal singularities,  this  was proved in  \cite[12.4.4]{km-flips}. The next result  extends these to all dimensions.

\begin{thm} \label{mmp.extends.conj.thm.c}
Let $g:X\to S$ be  a flat, proper morphism of reduced, complex analytic spaces. 
Fix a point $0\in S$ and assume that 
$X_0$ is  projective and  has canonical  singularities.
Then every sequence of MMP-steps
$X_0=X_0^0\map X_0^1\map X_0^2\map \cdots$ (see Definition~\ref{mmp.extends.conj.thm.defn})
extends to a sequence of MMP-steps  
$$
X=X^0\map X^1\map X^2\map \cdots,
$$
  over some open  neighborhood  $0\in U\subset S$.
\end{thm}

The proof is given in Paragraph~\ref{mmp.extends.conj.thm.c.pf} when $S$ is a disc $\dd$, and in Paragraph~\ref{mmp.extends.conj.thm.c.S.pf} in general.
The assumption that $X_0$ has  canonical singularities is necessary, as shown by semistable 3-fold flips \cite{km-flips}. Extending MMP steps from divisors with canonical singularities is also studied in \cite{k-amb}.

If $X_0$ is of general type, then
a suitable  MMP for $X_0$  terminates with a minimal model $X_0^{\rm m}$ by \cite{bchm}, which then extends to 
  $g^{\rm m}:X^{\rm m}_U\to U$ by Theorem~\ref{mmp.extends.conj.thm.c}. For  minimal models of varieties of general type, deformation invariance of plurigenera is easy, leading to a proof of (\ref{mmp.extends.conj.thm.c.cor}.1) in Paragraph~\ref{mmp.extends.conj.thm.c.cor.pf}. This also implies that all fibers are  bimeromorphic to a projective  variety.

If $X_0$ is smooth, then it is K\"ahler, and the $X_s$ are also  K\"ahler
by \cite{MR0112154}. A K\"ahler variety that is bimeromorphic to an algebraic variety is projective by \cite{Moi-66}. 

However, there are  families of surfaces with simple elliptic singularities
$g:X\to S$ such that $K_{X_0}$ is ample,
 all fibers are bimeromorphic to an algebraic surface, yet 
the projective fibers correspond to  a countable, dense set on the base; see Example~\ref{bad.fams.exmp}.

We use Theorem~\ref{12.2.10.thm.S}---taken from \cite[Thm.2]{k-sesh}---to obtain the projectivity of the fibers and
  complete the proof of
Theorem~\ref{mmp.extends.conj.thm.c.cor} in 
Paragraph~\ref{mmp.extends.conj.thm.c.cor.pf}.

\section{Examples  and consequences}

The first example shows that Theorem~\ref{mmp.extends.conj.thm.c.cor}
fails very badly for surfaces with non-canonical quotient singularities.

\begin{exmp} \label{amp.K.quot.sing.exmp}
We give an example of a flat, proper morphism of complex analytic spaces $g:X\to \dd$, such that
\begin{enumerate}
\item   $X_0$ is  a projective surface with a quotient singularity and ample canonical class, yet
\item  $X_s$ is smooth, non-algebraic, and of Kodaira dimension 0 for very general $s\in \dd$.
\end{enumerate}

Let us start with a K3 surface  $Y_0\subset \p^3$ with a hyperplane section $C_0\subset Y_0$ that is a rational curve with 3 nodes. We blow up the nodes
$Y'_0\to Y_0$ and  contract the birational transform of $C_0$ to get a surface $\tau_0:Y'_0\to X_0$. 
Let $E_1, E_2, E_3\subset X_0$ be the images of the 3 exceptional curves of the blow-up. 

By explicit computation,  we get a quotient singularity of type $\c^2/\frac18(1,1)$,  $(E_i^2)=-\frac12$ and $(E_i\cdot E_j)=\frac12$ for $i\neq j$. Furthermore,  $E:=E_1+E_2+ E_3\sim K_{X_0}$ and it is ample by the Nakai-Moishezon criterion.  (Note that $(E\cdot E_i)=\frac12$ and
$X_0\setminus E\cong Y_0\setminus C_0$ is affine.)

Take now a  deformation  $Y\to \dd$ of $Y_0$  whose very general fibers are non-algebraic K3 surfaces that contain no proper curves. Take 3 sections  $B_i\subset Y$ that pass through the 3 nodes of $C_0$. Blow them up and then contract the birational transform of $C_0$; cf.\ \cite{MR0304703}. In general \cite{MR0304703} says that the normalization of the resulting central fiber is $X_0$, but in our case 
the central fiber is isomorphic to $X_0$ since  $R^1(\tau_0)_*\o_{Y'_0}=0$. 
The contraction is an isomorphism on very  general fibers since there are no curves to contract. We get  $g:X\to \dd$ whose central fiber is $X_0$ and all other fibers are K3 surfaces blown up at 3 points.

In general, it is very unclear which complex varieties occur as deformations of projective varieties; see
\cite{kerr2021deformation} for some of their properties.
\end{exmp}

\begin{exmp} \label{surf.fam.exmp}\cite{Atiyah58}
 Let $S_0:=(g=0)\subset \p^3_{\mathbf x}$  and $S_1:=(f=0)\subset \p^3_{\mathbf x}$ be 
surfaces of the same degree. Assume that $S_0$ has only ordinary nodes,
$S_1$ is smooth,  $\pic(S_1)$ is generated by  the restriction of $\o_{\p^3}(1)$ and $S_1$ does not contain any of the singular points of $S_0$.  Fix $m\geq 2$ and consider
$$
X_m:=(g-t^mf=0)\subset \p^1_{\mathbf x}\times \a^1_t.
$$
The singularities are locally analytically of the form   $xy+z^2-t^m=0$. 
Thus $X_m$  is  locally analytically factorial if $m$ is odd.
If $m$ is even then  $X_m$ is factorial since the general fiber has Picard number 1, but it is not locally analytically factorial; blowing up $(x=z-t^{m/2}=0)$ gives a small resolution.  Thus we get that 
\begin{enumerate}
\item $X_m$ is  bimeromorphic to a proper,   smooth family of projective surfaces iff $m$ is even, but  
\item  $X_m$ is not bimeromorphic to a  smooth, projective family of  surfaces.
\end{enumerate}
\end{exmp}

\begin{exmp} \label{bad.fams.exmp}
Let $E\subset \p^2$ be a smooth cubic and take $r$ general lines  $L_i\subset \p^2$.  To get $S_0$, blow up all singular points of
$E+\tsum L_i$ and then contract the birational transform of $E+\tsum L_i$. 
A somewhat tedious computation shows that $K_{S_0}$ is ample for $r\geq 6$.
It has 1 simple elliptic singularity  (coming from $E$) and $r$ quotient singularities  (coming from the $L_i$).

Deform this example by moving the $3r$ points  $E\cap \tsum L_i$ into general position  $p^1_t, \dots, p^{3r}_t\in E$ and the points $L_i\cap L_j$  into general position on $\p^2$.
Blow up these points and then contract the birational transform of $E$
to get the surfaces $S_t$.  It has only 1 simple elliptic singularity  (coming from $E$).

We get a flat family of surfaces with central fiber $S_0$ and general fibers $S_t$.  Let $L$ denote the restriction of the line class on $\p^2$ to $E$. 

It is easy to see that 
such a surface $S_t$ is non-projective if the $p^i_t$ and $L$ are linearly independent in $\pic(E)$.
Thus   $S_t$ is not projective for very general $t$ and has Kodaira dimension 0.  

\end{exmp}

The next result is the scheme-theoretic version of Theorem~\ref{mmp.extends.conj.thm.c.cor}. Ideally it should be proved by the same argument. However, some of the references we use, especially
\cite{nak-book}, are worked out for analytic spaces, not for general schemes.
So for now we proceed in a somewhat roundabout way.

\begin{cor} \label{mmp.extends.alg.cor}
Let $S$ be a noetherian, excellent scheme over a  field of characteristic 0.
Let $g:X\to S$ be a flat, proper algebraic space. 
Fix a point $0\in S$ and assume that
$X_0$ is  projective, of general type and with  canonical  singularities.
Then   there is an open  neighborhood  $0\in S^\circ\subset S$ such that, for every $s\in S^\circ$,
\begin{enumerate}
\item the plurigenera  $h^0(X_s, \omega_{X_s}^{[r]})$ are independent of $s$ for every $r$, and
\item  the fiber $X_s$ is  projective.
\end{enumerate}
\end{cor}

\begin{proof} A proper algebraic space $Y$  over a field $k$ is projective iff 
$Y_K$ is  projective over $K$ for some field extension $K\supset k$. 
 Noetherian induction then shows that it is enough to prove the claims for the generic points of  the  completions    (at the  point $0\in S$) of irreducible subvarieties $0\in T\subset S$.
Since the defining equations of $\hat T$ and  of $X\times_S\hat T$ involve only countably many coefficients,  we may assume that the residue field is $\c$. 

 Consider now the local universal deformation space
$\defor(X_0)$  of $X_0$ in the complex analytic  category; see \cite{bing}. It is the germ of a complex analytic space and there is a complex analytic universal family
$
G: {\mathbf X}\to \defor(X_0).
$
Since a deformation over an Artin scheme is automatically  complex analytic, we see that
the formal completion
$
\hat G: \hat{\mathbf X}\to \widehat{\defor}(X_0)
$
is the universal formal deformation of $X_0$. 
In particular,  $X\times_S\hat T$ is the pull-back of
$ \hat G: \hat{\mathbf X}\to \widehat{\defor}(X_0)$
by a morphism  $\hat T\to \widehat{\defor}(X_0)$. 
Thus Theorem~\ref{mmp.extends.conj.thm.c.cor} implies both claims.
\end{proof}

\section{Relative MMP}

See \cite{km-book} for a general introduction to the minimal model program.

\begin{defn}[MMP-steps and their extensions] \label{mmp.extends.conj.thm.defn} 
Let $X\to S$ be a proper morphism of   complex analytic spaces with irreducible fibers.
Assume that $K_{X/S}$ is $\q$-Cartier.
By an {\it MMP-step} for  $X$ over $S$ we mean a diagram  
$$
\begin{array}{lcr}
X &\stackrel{\pi}{\map}&  X^+\\
\phi\searrow  && \swarrow \phi^+\\
& Z &
\end{array}
\eqno{(\ref{mmp.extends.conj.thm.defn}.1)}
$$
where all morphisms are bimeromorphic and proper over $S$, $-K_{X/S}$ is ample over $Z$,  $K_{X^+/S}$ is ample over $Z$ and
$\phi^+$ is small (that is, without exceptional divisors).

If $X$ is $\q$-factorial and  the relative Picard number of $X/Z$ is 1, then
 there are 2 possible  MMP steps:
\begin{itemize}
\item Divisorial:  $\phi$ contracts a single divisor and $\phi^+$  is the identity.
\item Flipping:  both $\phi$ and $\phi^+$ are small.
\end{itemize}
However, in general there is a more complicated possibility: 
\begin{itemize}
\item Mixed:  $\phi$ contracts (possibly several) divisors and $\phi^+$  is small.
\end{itemize}
 For our applications we only need to know that, by \cite[3.52]{km-book},  $X^+$ exists iff $\oplus_{r\geq 0}\ \omega_{Z/S}^{[r]} $  (which is equal to $\oplus_{r\geq 0} \phi_*\omega_{X/S}^{[r]} $) is a
finitely generated sheaf of $\o_Z$-algebras, and then 
$$
X^+=\proj_Z \oplus_{r\geq 0} \ \omega_{Z/S}^{[r]}.
\eqno{(\ref{mmp.extends.conj.thm.defn}.2)}
$$ 
We index a sequence of MMP-steps by setting $X^0:=X$ and 
$X^{i+1}:=(X^i)^+$.

Fix a point $s\in S$ and let
$X_s$ denote the fiber over $S$.  We say that a
sequence of MMP-steps (over $S$) $X^0\map X^1\map X^2\map \cdots$
{\it extends}   a sequence of MMP-steps  (over $s$)
$X_s^0\map X_s^1\map X_s^2\map \cdots$ if, for every $i$, 
$$
\begin{array}{rcl}
X_s^i\quad &\stackrel{\pi_s^i}{\map}& \quad X_s^{i+1}\\
\phi_s^i\searrow  && \swarrow(\phi_s^i)^+ \\
& Z_s^i &
\end{array}
\quad
\begin{array}{cc}
\mbox{is the fiber}\\
\mbox{over $s$ of}
\end{array}
\quad
\begin{array}{rcl}
X^i\quad &\stackrel{\pi^i}{\map}& \quad X^{i+1}\\
\phi^i\searrow  && \swarrow(\phi^i)^+ \\
& Z^i &
\end{array}
\eqno{(\ref{mmp.extends.conj.thm.defn}.3)}
$$
\end{defn}

\begin{say}[Proof of Theorem~\ref{mmp.extends.conj.thm.c} for $S=\dd$, the disc] \label{mmp.extends.conj.thm.c.pf}
 Since MMP-steps preserve canonical singularities, by induction it is enough to prove the claim for one MMP step. So we drop the upper index $i$ and identify $K_{X/\dd}$ with $K_X$.

Let $\phi_0:X_0\to Z_0$ be an extremal contraction.
By
\cite{MR0304703}\footnote{This should be changed to \cite[11.4]{km-flips}},
 it extends to  a contraction  $\phi:X\to Z$, where $Z$ is flat over $\dd$ with central fiber $Z_0$  since  $R^1(\phi_0)_*\o_{X_0}=0$. 
Note that  $K_X$ is $\q$-Cartier by (\ref{defs.of.can.lem}.1),  and 
$\phi$ is projective since $-K_X$ is $\phi$-ample. 

If $\phi_0$ is a divisorial contraction, then $K_{Z_0}$ is $\q$-Cartier,
and   so is  $K_Z$ by (\ref{defs.of.can.lem}.1). Thus $X^+=Z$. 

If $\phi_0$ is a flipping or mixed contraction, then
$K_Z$ is not $\q$-Cartier. By (\ref{mmp.extends.conj.thm.defn}.2),  
$$
X^+=\proj_Z \oplus_{r\geq 0} \ \omega_Z^{[r]},
\eqno{(\ref{mmp.extends.conj.thm.c.pf}.1)}
$$
provided $\oplus_{r\geq 0} \ \omega_Z^{[r]} $ is a
finitely generated sheaf of $\o_Z$-algebras.  (We have identified  $\omega_Z$ with $\omega_{Z/\dd}$.)

Functoriality works better if we twist by the line bundle $\o_Z(Z_0)$ and write it as
$$
X^+=\proj_Z \oplus_{r\geq 0} \ \omega_Z^{[r]}(rZ_0).
$$
Let $\tau: Y\to X$ be a projective  resolution of  $X$ (that is, $\tau$ is projective) such that $Y_0$, the bimeromorphic transform of $X_0$, is also smooth.
Set $g:=\phi\circ\tau$.

The hardest part of the proof is  Nakayama's theorem (\ref{nak.3.8}) which gives
 a surjection
$$
\oplus_{r\geq 0} g_*\omega_Y^{r}(rY_0)\onto 
\oplus_{r\geq 0} (g_0)_*\omega_{Y_0}^{r}.
\eqno{(\ref{mmp.extends.conj.thm.c.pf}.2)}
$$
Since $X_0$ has canonical singularities 
$\tau_*\omega_{Y_0}^{r}=\omega_{X_0}^{[r]}$, and hence
$g_*\omega_{Y_0}^{r}=\omega_{Z_0}^{[r]}$.
We also have a natural  inclusion
$g_*\omega_Y^{r}(rY_0)\into \omega_Z^{[r]}(rZ_0)$. 
Thus pushing forward (\ref{mmp.extends.conj.thm.c.pf}.2) we get a surjection
$$
\oplus_{r\geq 0} g_*\omega_Y^{r}(rY_0)\to 
\oplus_{r\geq 0} \ \omega_Z^{[r]}(rZ_0)\onto 
\oplus_{r\geq 0} \ \omega_{Z_0}^{[r]}.
\eqno{(\ref{mmp.extends.conj.thm.c.pf}.3)}
$$
Note that $\oplus_{r\geq 0} \ \omega_{Z_0}^{[r]} $ is a 
finitely generated sheaf of $\o_{Z_0}$-algebras, defining the
MMP-step of $X_0\to Z_0$.

Now (\ref{pt.fg.loc.fg.lem})  says that
$\oplus_{r\geq 0} \ \omega_Z^{[r]}(rZ_0) $ is also a
finitely generated sheaf of $\o_Z$-algebras, at least in some neighborhood of the  compact $Z_0$. \qed
\end{say}

Next we discuss various results used in the proof.

\begin{thm} \cite[VI.3.8]{nak-book} \label{nak.3.8}
Let $\pi:Y\to S$ be a projective, bimeromorphic  morphism of analytic spaces, $Y$ smooth and $S$ normal. 
Let $D\subset Y$ be a smooth, non-exceptional  divisor. Then the restriction map
$$
\pi_* \omega_Y^m(mD)\to \pi_* \omega_D^m \qtq{is surjective for $m\geq 1$.}\qed
$$
\end{thm}

 This is a special case of \cite[VI.3.8]{nak-book} applied with $\Delta=0$ and $L=K_Y+D$. 
\medskip

{\it Warning.}  The assumptions of  \cite[VI.3.8]{nak-book} are a little hard to find. They are outlined 11 pages earlier in
\cite[VI.2.2]{nak-book}. It talks about varieties, which usually suggest algebraic varieties, but \cite[p.231, line~13]{nak-book} explicitly states that the proofs  work with analytic spaces; see also \cite[p.14]{nak-book}.
(The statements of  \cite{nak-book} allow for a boundary $\Delta$. However, $K_Y+D+\Delta$ should be $\q$-linearly equivalent to a $\z$-divisor and $\rdown{\Delta}=0$ is assumed on \cite[p.231]{nak-book}. There seem to be  few cases when both of these can be satisfied.)

\begin{lem} \cite[VI.5.2]{nak-book} \label{defs.of.can.lem}
Let $g:X\to S$ be a flat morphism of complex analytic spaces.
Assume that $X_0$ has a canonical singularity at a point $x\in X_0$. 
Then there is an open neighborhood  $x\in X^*\subset X$ such that 
\begin{enumerate}
\item $K_{X^*/S}$ is $\q$-Cartier, and 
\item all fibers of $g|_{X^*}: X^*\to S$ have canonical singularities.
\end{enumerate}
\end{lem}

\begin{proof} (1)  is proved in
\cite[3.2.2]{k-thesis}; see also
\cite[12.7]{k-flat} and \cite[2.8]{k-modbook}.
The harder part is (2),  proved in  \cite[VI.5.2]{nak-book}.
\end{proof}
\medskip

{\it Remark \ref{defs.of.can.lem}.3.} If $S$ is smooth then $X^*$ has canonical singularities.
By induction, it is enough to prove this when $S=\dd$. 
Then the proof of \cite[VI.5.2]{nak-book} shows 
 that   even the pair  $(X^*, X_0\cap X^*)$
has canonical singularities.

\begin{lem} \label{pt.fg.loc.fg.lem}
Let $\pi:X\to S$  be a proper morphism of normal, complex spaces.
Let $L$ be a line bundle on $X$ and  $W\subset S$ a Zariski closed subset.
Assume that $\o_W\otimes_S \bigl(\oplus_{r\geq 0} \pi_* L^r\bigr)$ is a finitely generated sheaf of $\o_W$-algebras. 

Then every compact subset $W'\subset W$ has an open neighborhood
$W'\subset U\subset S$ such that $\o_U\otimes_S \bigl(\oplus_{r\geq 0} \pi_* L^r\bigr)$ is a finitely generated sheaf of $\o_U$-algebras. 
\end{lem}

\begin{proof}  The question is local on $S$, so we may as well assume  that
$W$ is a single point. 
We may also assume that $\o_W\otimes_S \bigl(\oplus_{r\geq 0} \pi_* L^r\bigr)$ is generated by $\pi_* L$.   After suitable blow-ups we are reduced to the case when the base locus of $L$ is a Cartier divisor $D$. By passing to a smaller neighborhood, we may assume that every irreducible component of $D$ intersects $\pi^{-1}(W)$.
By the Nakayama lemma, the  base locus of $L^r$ is a subscheme of $rD$ that agrees with it along $rD\cap \pi^{-1}(W)$. 
Thus  $rD$ is the  the base locus of $L^r$ for every $r$.
We may thus replace $L$ by $L(-D)$ and assume that $L$ is globally generated.

Thus $L$ defines a morphism  $X\to \proj_S\oplus_{r\geq 0} \pi_* L^r$,
let $\pi':X'\to S$ be its Stein factorization.  Then $L$ is the
pull-back of a  line bundle  $L'$ that is ample on  $X'\to S$ and
$\oplus_{r\geq 0} \pi_* L^r=\oplus_{r\geq 0} \pi'_* {L'}^r$
is finitely generated. \end{proof}

\begin{say}[Proof of Theorem~\ref{mmp.extends.conj.thm.c} for general $S$] \label{mmp.extends.conj.thm.c.S.pf}
  As in Paragraph~\ref{mmp.extends.conj.thm.c.pf}, it is enough to prove the claim for one MMP step, so let $\phi_0:X_0\to Z_0$ be an extremal contraction and  $\phi:X\to Z$ its extension. As before,  $Z$ is flat over $S$ with central fiber $Z_0$.

We claim that, for every $r$, 
\begin{enumerate}
\item $\omega_{Z/S}^{[r]}$ is flat over $S$, and
\item $\omega_{Z/S}^{[r]}|_{Z_0}\cong \omega_{Z_0}^{[r]}$.
\end{enumerate}
In the language of \cite{k-hh} or \cite[Chap.9]{k-modbook}, this says that
$\omega_{Z/S}^{[r]}$ is its own relative hull. 
There is an issue with precise references here, since \cite[Chap.9]{k-modbook}
is written in the algebraic setting. However, \cite[9.72]{k-modbook}
considers hulls over the spectra of complete local rings. Thus we get  that there is a unique largest subscheme
$\hat S^u\subset \hat S$ (the formal completion of $S$ at $0$) such that
(1--2) hold after base change to $\hat S^u$. 

By Paragraph~\ref{mmp.extends.conj.thm.c.pf} we know that 
(1--2) hold after base change to any disc $\dd\to S$, which implies that
$\hat S^u= \hat S$. That is, 
(1--2) hold for $\hat S$. 
Since both properties are invariant under formal completion, we are done.

Now we know that 
$$
X^+:=\proj_Z \oplus_{r\geq 0} \ \omega_{Z/S}^{[r]},
\eqno{(\ref{mmp.extends.conj.thm.c.S.pf}.3)}
$$
is flat over $S$ and its central fiber is $X^+_0$. 
Thus it gives the required extension of the flip of $X_0\to Z_0$. 
\qed
\end{say}

\section{Proof of Theorem~\ref{mmp.extends.conj.thm.c.cor}}

 We give a proof using only the $S=\dd$ case of 
Theorem~\ref{mmp.extends.conj.thm.c}.

\begin{say}
\label{mmp.extends.conj.thm.c.cor.pf}

Fix $r\geq 2$ and assume first that  $S=\dd$. Since $X_0$ is of general type, 
a suitable  MMP for $X_0$ ends with a minimal model  $X_0^{\rm m}$, and,
by Theorem~\ref{mmp.extends.conj.thm.c},  $X_0\map X_0^{\rm m}$ extends to a  fiberwise bimeromorphic map  $X\map X^{\rm m}$. 
We have $g^{\rm m}: X^{\rm m}\to \dd$. 
(From now on, we replace $\dd$ with a smaller disc whenever necessary.)
Since $K_{X_0^{\rm m}}$ is nef and big,
the higher cohomology groups of $\omega_{X_0}^{[r]}$ vanish for $r\geq 2$.
Thus $s\mapsto  H^0(X^{\rm m}_s,  \omega_{X^{\rm m}_s}^{[r]}) $
is locally constant at the origin.

By (\ref{defs.of.can.lem}.2)  $X_s$ and $X^{\rm m}_s $ both have canonical singularities, so they have the same plurigenera.  Therefore
 $s\mapsto  H^0(X_s,  \omega_{X_s}^{[r]}) $
is also locally constant  at the origin. By Serre duality, 
the deformation invariance of $H^0(X_s,  \omega_{X_s}) $
is equivalent to the deformation invariance of $H^n(X_s,  \o_{X_s}) $.
In fact, all the $H^i(X_s,  \o_{X_s}) $ are deformation invariant. 
For this the key idea is in
\cite{dub-jar}, which treats deformations of varieties with normal crossing singularities. The method works for varieties with canonical (even log canonical) singularities; this is worked out in \cite[Sec.2.5]{k-modbook}.

For arbitrary $S$, note that
  $s\mapsto  H^0(X_s,  \omega_{X_s}^{[r]} )$ is a constructible function on $S$, thus locally constant at $0\in S$ iff it is locally constant on every disc $\dd\to S$. 
Once $s\mapsto  H^0(X_s,  \omega_{X_s}^{[r]}) $ is locally constant at $0\in S$,
Grauert's theorem guarantees that 
$g_*\omega_{X/S}^{[r]}$ is locally free at $0\in S$ and commutes with base changes.

In principle it could happen that for each $r$ we need a smaller and smaller neighborhood, but  the same neighborhood works for all $r\geq 1$ by Lemma~\ref{pt.fg.loc.fg.lem}.

 Thus the plurigenera are deformation invariant,  all fibers are of general type, and   $g$ is fiberwise bimeromorphic to the relative canonical model
$$
X^{\rm c}:=\proj_{S} \oplus_{r\geq 0} g^{\rm m}_*\omega_{X^{\rm m}/S}^{[r]},
$$
which is projective over $S$. 
The projectivity of all fibers  now follows from the more precise  Theorem~\ref{12.2.10.thm.S}.
 \qed
\end{say}

The following is a special case of \cite[Thm.2]{k-sesh}.

\begin{thm} \label{12.2.10.thm.S}
  Let $g:X\to S$ be a  flat, proper morphism of complex analytic spaces
  whose fibers have rational singularities only. 
Assume that
 $g$ is bimeromorphic to a projective morphism $g^{\rm p}:X^{\rm p}\to S$,
 and  $X_0$ is projective for some $0\in S$. 

Then there is a Zariski  open neighborhood
$0\in U\subset S$ and a locally closed, Zariski stratification
$S=\cup_i S_i$ such that each
$$
g|_{X_i}:X_i:=g^{-1}(S_i)\to S_i \qtq{is projective.}\hfill\qed
$$
\end{thm}

\section{Open problems}

For deformations of varieties of general type, the following should be  true.

\begin{conj} Let $X_0$ be a projective variety of general type with canonical singularities. Then its universal deformation space  $\defor (X_0)$ has 
a representative  ${\mathbf X}\to S$ where $S$ is a scheme of finite type and
 ${\mathbf X}$ is an algebraic space.
\end{conj}

For  varieties of non-general type, the following is likely true
\cite[1.10]{rao-tsa}.

\begin{conj} \label{mmp.extends.conj.thm.c.cor.ng}
Let $g:X\to S$ be a flat, proper morphism of complex analytic spaces. Assume that
$X_0$ is  projective and with  canonical  singularities.
Then  
 the plurigenera  $h^0(X_s, \omega_{X_s}^{[r]})$ are independent of $s\in S$ for every $r$, in some neighborhood of $0\in S$. 
\end{conj}

{\it Comments.} One can try to follow the proof of 
Theorem~\ref{mmp.extends.conj.thm.c.cor}.
If $X_0$ is not of general type, we run into several difficulties in relative dimensions $\geq 4$.
MMP is not know to terminate and even if we get a minimal model,
abundance is not known. If we have a good minimal model, 
then we run into the  following.

\begin{conj} \label{pluri.inv.abound.conj}
Let $X$ be a complex space and
 $g:X\to S$  a flat, proper morphism. Assume that
$X_0$  is projective,  has canonical  singularities and $\omega_{X_0}^{[r]}$ is globally generated for some $r>0$. 
Then the plurigenera are locally constant at $0\in S$.
\end{conj}

{\it Comments.} More generally, the same may hold if $X_0$ is  Moishezon (that is, bimeromorphic to a projective variety),  K\"ahler or  in Fujiki's class $\mathcal C$ (that is, bimeromorphic to a compact K\"ahler manifold; see \cite{ueno-83} for an introduction).

A positive answer is known in many cases. 
\cite[12.5.5]{km-flips} proves this
if $X_0$ is projective and has terminal singularities. However, 
the proof works for the Moishezon and class $\mathcal C$ cases as well.

The projective case with  canonical  singularities is discussed in  \cite[VI.3.15--16]{nak-book}; I believe that  the projectivity assumption is very much built into the proof given there; see \cite[VI.3.11]{nak-book}. 
\medskip

\begin{ack} I thank D.~Abramovich, F.~Campana, J.-P.~Demailly, O.~Fujino, A.~Landesman, S.~Mori,  T.~Murayama, V.~Tosatti, D.~Villalobos-Paz, C.~Voisin and C.~Xu
for  helpful comments and   corrections.
Partial  financial support    was provided  by  the NSF under grant number
DMS-1901855.
\end{ack}


\begin{thebibliography}{BCHM10}

\bibitem[AK19]{k-amb}
Florin Ambro and J{\'a}nos Koll{\'a}r, \emph{Minimal models of
  semi-log-canonical pairs}, Moduli of K-stable Varieties, Springer INdAM Ser.,
  vol.~31, Springer, Cham, 2019, pp.~1--13.

\bibitem[Ati58]{Atiyah58}
M.~F. Atiyah, \emph{On analytic surfaces with double points}, Proc. Roy. Soc.
  London. Ser. A \textbf{247} (1958), 237--244. \MR{MR0095974 (20 \#2472)}

\bibitem[BCHM10]{bchm}
Caucher Birkar, Paolo Cascini, Christopher~D. Hacon, and James
  M\textsuperscript{c}Kernan, \emph{Existence of minimal models for varieties
  of log general type}, J. Amer. Math. Soc. \textbf{23} (2010), no.~2,
  405--468.

\bibitem[Bin87]{bing}
J\"urgen Bingener, \emph{Lokake {M}odulr\"aume in den analytischen {G}eometrie
  {I}--{II}}, Aspects of math., vol. 302, Viehweg, Braunschweig, 1987.

\bibitem[DJ74]{dub-jar}
Philippe Dubois and Pierre Jarraud, \emph{Une propri\'et\'e de commutation au
  changement de base des images directes sup\'erieures du faisceau structural},
  C. R. Acad. Sci. Paris S\'er. A \textbf{279} (1974), 745--747. \MR{0376678
  (51 \#12853)}

\bibitem[Iit69]{Iitaka69}
Shigeru Iitaka, \emph{Deformations of compact complex surfaces {I}}, Global
  Analysis (Papers in Honor of K. Kodaira), Univ. Tokyo Press, Tokyo, 1969,
  pp.~267--272. \MR{MR0260746 (41 \#5369)}

\bibitem[KLS21]{kerr2021deformation}
Matt Kerr, Radu Laza, and Morihiko Saito, \emph{Deformation of rational
  singularities and {H}odge structure}, 2021.

\bibitem[KM92]{km-flips}
J{\'a}nos Koll{\'a}r and Shigefumi Mori, \emph{Classification of
  three-dimensional flips}, J. Amer. Math. Soc. \textbf{5} (1992), no.~3,
  533--703. \MR{1149195 (93i:14015)}

\bibitem[KM98]{km-book}
\bysame, \emph{Birational geometry of algebraic varieties}, Cambridge Tracts in
  Mathematics, vol. 134, Cambridge University Press, Cambridge, 1998, With the
  collaboration of C. H. Clemens and A. Corti, Translated from the 1998
  Japanese original.

\bibitem[Kol83]{k-thesis}
J{\'a}nos Koll{\'a}r, \emph{Toward moduli of singular varieties, {P}h.{D}.
  thesis}, 1983.

\bibitem[Kol95]{k-flat}
\bysame, \emph{Flatness criteria}, J. Algebra \textbf{175} (1995), no.~2,
  715--727. \MR{1339664 (96j:14010)}

\bibitem[Kol08]{k-hh}
\bysame, \emph{Hulls and husks}, arXiv:0805.0576, 2008.

\bibitem[Kol21a]{k-modbook}
\bysame, \emph{Moduli of varieties of general type}, (book in preparation,
  \url{https://web.math.princeton.edu/~kollar/FromMyHomePage/modbook.pdf}),
  2021.

\bibitem[Kol21b]{k-sesh}
\bysame, \emph{Seshadri's criterion and openness of projectivity}, math.AG:
  2105.06242, 2021.

\bibitem[KS58]{MR0112154}
K.~Kodaira and D.~C. Spencer, \emph{On deformations of complex analytic
  structures. {I}, {II}}, Ann. of Math. (2) \textbf{67} (1958), 328--466.
  \MR{MR0112154 (22 \#3009)}

\bibitem[Moi66]{Moi-66}
Boris Moishezon, \emph{On $n$-dimensional compact varieties with $n$
  algebraically independent meromorphic functions, {I}, {II} and {III} (in
  {R}ussian)}, Izv. Akad. Nauk SSSR Ser. Mat. \textbf{30} (1966), 133--174,
  345--386, 621--656.

\bibitem[MR71]{MR0304703}
A.~Markoe and H.~Rossi, \emph{Families of strongly pseudoconvex manifolds},
  Symposium on {S}everal {C}omplex {V}ariables ({P}ark {C}ity, {U}tah, 1970),
  Springer, Berlin, 1971, pp.~182--207. Lecture Notes in Math., Vol. 184.
  \MR{0304703 (46 \#3835)}

\bibitem[Nak04]{nak-book}
Noboru Nakayama, \emph{Zariski-decomposition and abundance}, MSJ Memoirs,
  vol.~14, Mathematical Society of Japan, Tokyo, 2004.

\bibitem[RT20]{rao-tsa}
Sheng Rao and I-Hsun Tsai, \emph{Invariance of plurigenera and {C}how-type
  lemma}, math.AG: 2011.03306, 2020.

\bibitem[Siu98]{siu-plur}
Yum-Tong Siu, \emph{Invariance of plurigenera}, Invent. Math. \textbf{134}
  (1998), no.~3, 661--673. \MR{1660941 (99i:32035)}

\bibitem[Uen83]{ueno-83}
Kenji Ueno, \emph{Introduction to the theory of compact complex spaces in the
  class {$\mathcal C$}}, Algebraic Varieties and Analytic Varieties, Advanced
  Studies in Pure Mathematics, vol.~1, 1983, pp.~219--230.

\end{thebibliography}

\def\cprime{$'$} \def\cprime{$'$} \def\cprime{$'$} \def\cprime{$'$}
  \def\cprime{$'$} \def\dbar{\leavevmode\hbox to 0pt{\hskip.2ex
  \accent"16\hss}d} \def\cprime{$'$} \def\cprime{$'$}
  \def\polhk#1{\setbox0=\hbox{#1}{\ooalign{\hidewidth
  \lower1.5ex\hbox{`}\hidewidth\crcr\unhbox0}}} \def\cprime{$'$}
  \def\cprime{$'$} \def\cprime{$'$} \def\cprime{$'$}
  \def\polhk#1{\setbox0=\hbox{#1}{\ooalign{\hidewidth
  \lower1.5ex\hbox{`}\hidewidth\crcr\unhbox0}}} \def\cdprime{$''$}
  \def\cprime{$'$} \def\cprime{$'$} \def\cprime{$'$} \def\cprime{$'$}
\providecommand{\bysame}{\leavevmode\hbox to3em{\hrulefill}\thinspace}
\providecommand{\MR}{\relax\ifhmode\unskip\space\fi MR }
\providecommand{\MRhref}[2]{%
  \href{http://www.ams.org/mathscinet-getitem?mr=#1}{#2}
}
\providecommand{\href}[2]{#2}

\bigskip

  Princeton University, Princeton NJ 08544-1000, \

\email{kollar@math.princeton.edu}

\end{document}